\documentclass[11pt]{article}
\usepackage{amscd, amsmath, amssymb, amsthm, tikz}
\usepackage[all,cmtip]{xy}
\usepackage[pagebackref]{hyperref}

\title{The failure of Kodaira vanishing for Fano varieties,
and terminal singularities that are not Cohen-Macaulay}
\author{Burt Totaro}
\date{  }

\def\Z{\text{\bf Z}}
\def\Q{\text{\bf Q}}

\def\P{\text{\bf P}}
\def\F{\text{\bf F}}
\def\N{\text{\bf N}}

\def\arrow{\rightarrow}

\DeclareMathOperator{\Hom}{Hom}
\DeclareMathOperator{\Pic}{Pic}
\DeclareMathOperator{\Spec}{Spec}
\def\red{\text{red}}
\def\supp{\text{supp}}
\def\Fl{\text{Fl}}
\def\Gr{\text{Gr}}

\begin{document}
\maketitle
\newtheorem{theorem}{Theorem}[section]
\newtheorem{corollary}[theorem]{Corollary}
\newtheorem{lemma}[theorem]{Lemma}

\theoremstyle{definition}
\newtheorem{definition}[theorem]{Definition}
\newtheorem{example}[theorem]{Example}

\theoremstyle{remark}
\newtheorem{remark}[theorem]{Remark}

The Kodaira vanishing theorem says
that for an ample line bundle $L$ on a smooth
projective variety $X$ over a field of characteristic zero,
$$H^i(X,K_X+L)=0$$
for all $i>0$. (Here $K_X$ denotes the canonical bundle, and we use
additive notation for line bundles.) This result and its
generalizations are central to the classification
of algebraic varieties. For example,
Kodaira vanishing can sometimes be used to show
that there are only finitely many varieties
with given intrinsic invariants, up to deformation equivalence.
Unfortunately, Raynaud showed that
Kodaira vanishing fails already for surfaces in any
characteristic $p>0$ \cite{Raynaud}.

For the minimal model program (MMP) in positive characteristic,
it has been important to find out whether Kodaira vanishing holds
for special classes of varieties, notably for Fano varieties (varieties
with $-K_X$ ample).
By taking cones, this is related to the question of whether
the singularities arising in the MMP (klt, canonical,
and so on) have the good properties (such as Cohen-Macaulayness
or rational singularities) familiar from characteristic zero.

For example, Kodaira vanishing holds for smooth del Pezzo surfaces
in any characteristic \cite[Theorem II.1.6]{Ekedahl}.
Also, klt surface singularities in characteristic $p>5$ are
strongly F-regular and hence Cohen-Macaulay;
this is the key reason why the MMP
for 3-folds is only known in characteristic $p>5$ (or zero)
\cite[Theorem 3.1]{HX}. There are in fact
some striking counterexamples in characteristics 2 and 3.
Maddock found a regular (but not smooth) del Pezzo surface $X$ over
an imperfect field of characteristic 2 with $H^1(X,O)\neq 0$,
which violates Kodaira vanishing \cite{Maddock}.
And Cascini-Tanaka and Bernasconi found klt
3-folds over algebraically closed fields
of characteristic 2 or 3 which are not Cohen-Macaulay \cite[Theorem 1.3]{CT},
\cite[Theorem 1.2]{Bernasconi}.

So far, the only known example of a smooth Fano variety on which
Kodaira vanishing fails has been a 6-fold in characteristic 2
discovered by Haboush and Lauritzen \cite{HL, LR}.
In this paper, however, we find
that Kodaira vanishing fails for smooth Fano varieties
in every characteristic $p>0$. One family of examples
has dimension 5 for $p=2$
and $2p-1$ for $p\geq 3$
(Theorem \ref{van2p-1}).
(Thus, in characteristics 2 and 3, this is a 5-fold
rather than a 6-fold.)

Our examples are projective homogeneous varieties
with non-reduced stabilizer, as in the Haboush-Lauritzen example.
Projective homogeneous varieties are smooth and rational
in any characteristic $p>0$,
but most of them are not Fano (apart from the familiar flag varieties,
where the stabilizer subgroup is reduced).
A point that seems to have been overlooked
is that certain infinite families of ``nontrivial''
homogeneous varieties are Fano.
We disprove Kodaira vanishing for some of these varieties.

Kov\'acs found
that for the Haboush-Lauritzen Fano 6-fold $X$ in characteristic 2,
Kodaira vanishing fails already for the ample line bundle $-2K_X$;
explicitly, we have $H^1(X,-K_X)\neq 0$ \cite{Kovacs}.
By taking a cone over $X$,
he gave an example of a canonical singularity which is not
Cohen-Macaulay, on a 7-fold in characteristic 2. Yasuda then
constructed quotient singularities of any characteristic $p>0$
which are canonical but not Cohen-Macaulay \cite[Remark 5.3]{HW}.

We find an even better phenomenon
among Fano varieties in every characteristic
$p>2$. Namely, for every $p>2$, there is a smooth Fano variety $X$
in characteristic $p>0$
such that $-K_X$ is divisible by 2, $-K_X=2A$, and Kodaira vanishing
fails for the ample line bundle $3A$; explicitly, we have
$H^1(X,A)\neq 0$ (Theorem \ref{van2p+1}).
Here $X$ has dimension $2p+1$. By taking a cone
over $X$, we give a first example of a terminal singularity
which is not Cohen-Macaulay. Moreover,
we have such examples in every characteristic
$p>2$.

After these results were announced, Takehiko Yasuda showed
that there are also quotient singularities of any characteristic
$p>0$ which are terminal but not Cohen-Macaulay (section
\ref{quotient}). Inspired by Yasuda's examples (which are quotients
by a finite group acting linearly), we construct
a new example in the lowest possible dimension: a terminal
singularity {\it of dimension 3 }over $\F_2$ which is not
Cohen-Macaulay (Theorem \ref{dim3}). Our singularity is the quotient
of a smooth variety by a non-linear action of the group $\Z/2$,
and such quotients should be a rich source of further examples.
The paper concludes with some open questions.

I thank Omprokash Das, John Ottem,
and Takehiko Yasuda for useful discussions.
This work was supported by NSF grant DMS-1701237.

\section{Projective homogeneous varieties with non-reduced
stabilizer group}
\label{homvar}

In this section, we describe the projective homogeneous
varieties with non-reduced stabilizer group. More details
are given by Haboush and Lauritzen \cite{HL}.

Let $G$ be a simply connected split semisimple group over a field
$k$ of characteristic $p>0$. Let $T$ be a split maximal torus
in $G$ and $B$ a Borel subgroup containing $T$. Let $\Phi\subset X(T)$
be the set of roots of $G$ with respect to $T$, and define
the subset $\Phi^{+}$ of positive roots to
be the roots of $G$ that are not roots of $B$.
Let $\Delta$ be the associated set of simple roots.
Choose a numbering of the simple roots,
$\Delta=\{\alpha_1,\ldots,\alpha_l\}$, where $l$ is the rank of $G$
(the dimension of $T$). For each
root $\beta$, there is an associated root subgroup
$U_{\beta}\subset G$ isomorphic to the additive group $G_a$,
the unique $G_a$ subgroup normalized
by $T$ on which $T$ acts by $\beta$.

Let $f$ be a function from $\Delta$ to the set of natural numbers
together with $\infty$. The function $f$ determines
a subgroup scheme $P$ of $G$ containing $B$, as follows.
Every positive root $\beta$ can be written uniquely as a linear
combination of simple roots with nonnegative coefficients.
The {\it support }of $\beta$ means the set of simple roots
whose coefficient in $\beta$ is positive.
Extend the function $f\colon \Delta\arrow \N\cup{\infty}$ to a
function on all the positive roots, by defining
$$f(\beta)=\inf_{\alpha\in \supp(\beta)} f(\alpha).$$

For a natural number $r$,
let $\alpha_{p^r}$ be the subgroup scheme of $G_a$ defined
by $x^{p^r}=0$. Let $P_{\red}$ be the parabolic subgroup
generated by $B$ and the root subgroup $U_{\beta}$ for each simple root
$\beta$ with $f(\beta)=\infty$. Finally, define the subgroup scheme
$P$ (with underlying reduced subgroup $P_{\red}$)
as the product (in any order) of $P$ and the subgroup
scheme $\alpha_{p^r}$ of $U_{\beta}$ for each positive root
$\beta$ with $f(\beta)=r<\infty$. Wenzel showed that if $p\geq 5$,
or if $G$ is simply laced, then
every subgroup scheme of $G$ containing $B$ is of this form
for some function $f$ \cite[Theorem 14]{Wenzel}.
For our purpose, it is enough to use these examples of subgroup
schemes.

The natural surjection
$G/P_{\red}\arrow G/P$ is finite and inseparable. The homogeneous
variety $G/P$
is smooth and projective over $k$ (even though $P$ is not smooth).
Lauritzen showed that
$G/P$ has a cell decomposition over $k$, coming from the Bruhat
decomposition of $G/P_{\red}$ \cite{Lauritzencell}. In particular,
$G/P$ is rational over $k$, and $H^i(G/P,\Omega^j)=0$ for $i\neq j$.

If the function $f$ takes only one value $r$ apart from $\infty$, then $P$
is the subgroup scheme generated by $P_{\red}$ and the $r$th Frobenius
kernel of $G$. In that case, $G/P$ is isomorphic to $G/P_{\red}$
as a variety.
(In terms of such an isomorphism,
the surjection $G/P_{\red}\arrow G/P$ is the $r$th power
of the Frobenius endomorphism of $G/P_{\red}$.)
By contrast, for more general functions $f$,
there can be intriguing differences between the properties of $G/P$
and those of the familiar flag variety $G/P_{\red}$.

The Picard group $\Pic(G/P)$ can be identified with a subgroup
of $\Pic(G/P_{\red})$, or of $\Pic(G/B)=X(T)$, by pullback.
For each root $\alpha$ in $X(T)$, write $\alpha^{\vee}$
for the corresponding coroot in $Y(T)=\Hom(X(T),\Z)$.
Then $X(T)$ has a basis given by the fundamental weights
$\omega_1,\ldots,\omega_l$, which are characterized by:
$$\langle \omega_i,\alpha_j^{\vee}\rangle=\delta_{ij}.$$
The subgroup $\Pic(G/P_{\red})$ is the subgroup generated by the $\omega_i$
with $f(i)<\infty$. And $\Pic(G/P)$ is the subgroup generated
by $p^{f(i)}\omega_i$, for each $i$ with $f(i)<\infty$
\cite[section 2, Corollary 7]{HL}. A line bundle on $G/P$
is ample if and only if its pullback to $G/P_{\red}$ is ample,
which means that the coefficient of $\omega_i$ is positive
for every $i\in\{1,\ldots,l\}$ with $f(i)<\infty$.
Moreover, Lauritzen showed that every ample line bundle
on a homogeneous variety $G/P$ is very ample (that is, it has enough
sections to embed $G/P$ into projective space)
\cite[Theorem 1]{Lauritzenemb}.

The anticanonical bundle of $X=G/P$ is given by
\cite[section 3, Proposition 7]{HL}
$$-K_X=\sum_{\substack{\beta\in\Phi^{+}\\ f(\beta)<\infty}}p^{f(\beta)}\beta.$$

Haboush and Lauritzen showed that $G/P$ is never Fano when
$P_{\red}=B$ and $p\geq 5$, except when the function $f$
is constant, in which case $X$ is isomorphic
to the full flag variety $G/B$
\cite[section 4, proof of Theorem 3]{HL}.
By contrast, inseparable images of some partial flag varieties
can be Fano without being isomorphic to the partial flag variety.
The example we use in this paper is:
let $G=SL(n)$ for $n\geq 3$ over a field $k$ of characteristic
$p>0$. As is standard, write the weight lattice
as $X(T)=\Z\{L_1,\ldots,L_n\}
/(L_1+\cdots+L_n=0)$ \cite[section 15.1]{FH}.
The fundamental weights are given by $\omega_i=L_1+\cdots+L_i$,
and the simple roots are $\alpha_i=L_i-L_{i+1}$,
for $1\leq i\leq n-1$. Thus $\alpha_i=-\omega_{i-1}+2\omega_i-\omega_{i+1}$
for $1\leq i\leq n-1$, with the convention that $\omega_0$ and $\omega_n$
are zero.
Let $P$ be the subgroup scheme associated to the function
$$f(i)=\begin{cases} 1&\text{if }i=1.\\
0&\text{if }i=2\\
\infty &\text{if }3\leq i\leq n-1.
\end{cases}$$
Then $G/P_{\red}$ is the flag manifold $\Fl(1,2,n)$, of dimension
$2n-3$ and Picard number 2,
and so $X:=G/P$ also has dimension $2n-3$ and Picard number 2.
By the formula above, $X$ has anticanonical bundle
\begin{align*}
-K_X&=p(L_{1}-L_2)+\sum_{i=3}^{n}(L_2-L_i)
+\sum_{i=3}^{n}(L_1-L_i)\\
&=(p+n-2)L_1+(-p+n-2)L_2-2(L_3+\cdots+L_{n})\\
&=2p\, \omega_1+(n-p)\omega_{2}.
\end{align*}
Thus $-K_X$ is ample if and only if $2p>0$ and $n-p>0$,
which means that $p<n$.

\section{Terminal cones that are not Cohen-Macaulay}

\begin{theorem}
\label{van2p+1}
Let $p$ be a prime number at least 3. Then there is a smooth
Fano variety $X$ over $\F_p$ such that $-K_X$ is divisible by 2,
$-K_X=2A$, and $H^1(X,A)\neq 0$. Here $X$ has dimension $2p+1$
and Picard number 2.

Moreover, for $p\geq 5$, the Euler characteristic $\chi(X,A)$ is negative.
\end{theorem}

Thus Kodaira vanishing fails for the ample line bundle $3A$.
These are the first examples of smooth Fano varieties in characteristic
greater than 2 for which Kodaira vanishing fails. Haboush and Lauritzen
exhibited a smooth Fano 6-fold in characteristic 2 for which Kodaira vanishing
fails. (This is \cite[section 6, Example 4]{HL} or \cite[section 2]{LR}.
Both papers give examples of the failure of Kodaira vanishing
in any characteristic, but the variety they consider is Fano
only in characteristic 2.)

In the cases where $\chi(X,A)$ is negative, one amusing consequence
is that the smooth Fano variety $X$ does not lift to characteristic
zero, even over a ramified extension of the $p$-adic integers.
Indeed, given such a lift ${\cal X}$,
$A$ would lift to an ample line bundle on ${\cal X}$ by Grothendieck's
algebraization theorem \cite[Tag 089A]{Stacks}, because
$H^2(X,O)=0$ in this example. But then $\chi(X,A)$ would be nonnegative,
by Kodaira vanishing in characteristic zero.

\begin{corollary}
\label{cone}
Let $p$ be a prime number at least 3. Then there is an isolated terminal
singularity over $\F_p$ which is not Cohen-Macaulay. One can take
its dimension to be $2p+2$.
\end{corollary}

These were the first examples of terminal singularities
that are not Cohen-Macaulay. After these results
were announced, Yasuda gave lower-dimensional examples,
described in section \ref{quotient}. That in turn
inspired the 3-dimensional example in this paper (Theorem \ref{dim3}).

\begin{proof} (Corollary \ref{cone})
In the notation of Theorem \ref{van2p+1},
let $Y$ be the affine cone over the smooth Fano variety $X$ with respect
to the ample line bundle $A$, meaning
$$Y=\Spec \oplus_{m\geq 0}H^0(X,mA).$$
A cone $Y$ is terminal if and only if the ample line bundle $A$
is $\Q$-linearly equivalent
to $a(-K_X)$ for some $0<a<1$, as is the case here (with $a=1/2$)
\cite[3.1]{Kollarsing}.
Also, a cone $Y$ is Cohen-Macaulay if and only if $H^i(X,mA)=0$
for all $0<i<\dim(X)$ and all $m\in \Z$ \cite[3.11]{Kollarsing}.
Since $H^1(X,A)\neq 0$
and $X$ has dimension $2p+1>1$,
$Y$ is not Cohen-Macaulay.
\end{proof}

\begin{proof} (Theorem \ref{van2p+1})
Let $n=p+2$, and let $X=G/P$ be the smooth projective
homogeneous variety for $G=SL(n)$
over $\F_p$ associated to the function $f$ from
section \ref{homvar}. Thus $X$ has dimension $2n-3=2p+1$, and
\begin{align*}
-K_X&=2p\, \omega_1+(n-p)\omega_{2}\\
&=2p\, \omega_1+2\omega_{2}.
\end{align*}
Because $2p$ and $2$ are positive,
$-K_X$ is ample. Because $\Pic(X)$ has a basis given
by $p\, \omega_1$ and $\omega_2$, $-K_X$ is divisible by 2
in $\Pic(X)$,
$-K_X=2A$, with $A:=p\,\omega_1+\omega_{2}$. As discussed
in section \ref{homvar}, the ample line bundle $A$ on $X$ is in fact
very ample.

In the notation of section \ref{homvar},
let $Q$ be the parabolic subgroup of $G$ associated to the function
$$h(i)=\begin{cases} \infty &\text{if }1\leq i\leq n-1\text{ and }i\neq 2\\
0 &\text{if }i=2.
\end{cases}$$
Because $f\leq h$,
$P$ is contained in $Q$. The morphism $G/P\arrow G/Q$
is a $\P^1$-bundle, and $G/Q$ is the Grassmannian $\Gr(2,n)$.
In more detail, $G/P$ is the Frobenius twist of the obvious
$\P^1$-bundle over this Grassmannian. We analyze the cohomology
of $A$ using this $\P^1$-bundle, as Haboush
and Lauritzen did in a similar situation \cite[section 6]{HL}.

Write $\alpha$ for the simple root $\alpha_{1}$, the one with
$f(\alpha)=1$.
For any line bundle on $G/P$, identified with a weight
$\lambda$, $\langle \lambda,\alpha^{\vee}\rangle$ is a multiple of $p$,
and the degree of $\lambda$ on the $\P^1$ fibers of $G/P\arrow G/Q$
is $\langle \lambda,\alpha^{\vee}\rangle/p$.
For the line bundle $A$, we have
$\langle A,\alpha^{\vee}\rangle=p$, and so $A$ has degree 1 on the fibers
of $G/P\arrow G/Q$.

Consider the commutative diagram
$$\xymatrix@C-10pt@R-10pt{
G/B\ar[r]\ar[rd]_{\gamma} & G/P\ar[d]_{\pi}\\
& G/Q.
}$$
Since the line bundle $A$ has degree 1 on the $\P^1$ fibers
of $G/P\arrow G/Q$, it has no higher cohomology on the fibers,
and so $H^i(G/P,A)\cong H^i(G/Q,\pi_*(A))$ for all $i$. Moreover,
since $h^0(\P^1,O(1))=2$, $\pi_*(A)$ is a vector bundle
of rank 2 on $G/Q$.
Next, the morphism $\gamma\colon G/B\arrow G/Q$ has fibers
isomorphic to $Q/B$, which is a projective homogeneous variety
(with reduced stabilizer groups); in the case at hand, this fiber
is $\P^1$ times the flag manifold of $SL(n-2)$. Therefore,
$H^i(Q/B,O)=0$ for $i>0$, and hence $R\gamma_*(O_{G/B})=
O_{G/Q}$. By the projection formula, it follows
that $H^i(G/Q,\pi_*(A))\cong H^i(G/B,\gamma^*\pi_*A)$
for all $i$. This and the previous isomorphism are isomorphisms
of $G$-modules.

Finally, $\gamma^*\pi_*A$ is a $G$-equivariant rank-2 vector bundle
on $G/B$, and so it corresponds to a 2-dimensional representation $M$ of $B$.
Every representation of $B$ is an extension of 1-dimensional
representations $k(\mu)$, corresponding to weights $\mu$ of $T$. In this case,
a direct computation \cite[section 6, Proposition 2]{HL}
shows that $M$ is an extension
$$0\arrow k(\lambda-p\alpha)\arrow M\arrow k(\lambda)\arrow 0,$$
where $\lambda:=p\,\omega_1+\omega_{2}$ is the weight corresponding
to the line bundle $A$.

As a result, the long exact sequence of cohomology on $G/B$
takes the form:
$$0\arrow H^0(G/B,\lambda-p\alpha)\arrow H^0(G/P,A)\arrow H^0(G/B,\lambda)
\arrow H^1(G/B,\lambda-p\alpha)\arrow \cdots.$$
At this point, we could compute the Euler characteristic
$\chi(G/P,A)=\chi(G/B,\lambda)+\chi(G/B,\lambda-p\alpha)$
by the Riemann-Roch theorem on $G/B$ (essentially the Weyl
character formula), and see that $\chi(G/P,A)$ is negative
for $p\geq 5$. That would suffice to disprove Kodaira vanishing
for $A$. However, we choose to give a more detailed analysis
of the cohomology of $A$, which will apply to the case $p=3$ as well.

Let $\rho\in X(T)$ be half the sum of the positive roots;
this is also the sum of the fundamental weights.
As is standard in Lie theory, consider
the ``dot action'' of the Weyl group $W$ of $G$
on the weight lattice:
$$w\cdot \mu=w(\mu+\rho)-\rho.$$
Let $s_{\beta}\in W$ be the reflection associated to a root
$\beta$.

We use the following result of Andersen on the cohomology
of line bundles on the flag variety
\cite[Proposition II.5.4(d)]{Jantzen}:

\begin{theorem}
\label{andersen}
For a reductive group $G$ in characteristic $p>0$,
a simple root $\beta$, and a weight $\mu$ with $\langle
\mu,\beta^{\vee}\rangle$ of the form $sp^m-1$ for
some $s,m\in \N$ with $0<s<p$, there is an isomorphism
of $G$-modules:
$$H^i(G/B,\mu)\cong H^{i+1}(G/B,s_{\beta}\cdot\mu).$$
\end{theorem}

For the weight $\lambda$ and simple root $\alpha$ considered above,
we have $\langle \lambda-\alpha,\alpha^{\vee}\rangle
=p-2$ and $s_{\alpha}\cdot (\lambda-\alpha)=\lambda-p\alpha$.
Here $p-2$ is in the range where Theorem \ref{andersen} applies,
and so we have
$$H^i(G/B,\lambda-\alpha)\cong H^{i+1}(G/B,s_{\alpha}\cdot(\lambda
-\alpha)$$
Explicitly, $\lambda-\alpha=(p-2)\omega_1+2\omega_{2}$,
which is a dominant weight (like $\lambda$). By Kempf's vanishing theorem,
it follows that the line bundles $\lambda$ and $\lambda-\alpha$
on $G/B$ have cohomology concentrated in degree zero
\cite[Proposition II.4.5]{Jantzen}. Assembling all these results,
we have an exact sequence of $G$-modules:
$$0\arrow H^0(G/P,A)\arrow H^0(G/B,\lambda)\arrow
H^0(G/B,\lambda-\alpha)\arrow H^1(G/P,A)\arrow 0,$$
and $H^i(G/P,A)=0$ for $i\geq 2$.

In view of Kempf's vanishing theorem, the dimensions of
the {\it Schur modules }$H^0(G/B,\lambda)$
and $H^0(G/B,\lambda-\alpha)$ are given by the Weyl dimension formula,
as in characteristic zero. For $SL(n)$,
the formula says \cite[Theorem 6.3(1)]{FH}:
for a dominant weight
$\mu=a_1\omega_1+\cdots+a_{n-1}\omega_{n-1}$,
$$h^0(G/B,\mu)=\prod_{1\leq i<j\leq n}\frac{a_i+\cdots+a_{j-1}+j-i}
{j-i}.$$
We read off that
$$h^0(G/B,\lambda)=\binom{2p+2}{p}(p+1)$$
and
$$h^0(G/B,\lambda-\alpha)=\binom{2p+1}{p}\frac{(p+2)(p-1)}{2}.$$
To compare these numbers, compute the ratio:
$$\frac{h^0(G/B,\lambda-\alpha)}{h^0(G/B,\lambda)}
=\frac{(p-1)(p+2)^2}{4(p+1)^2}.$$
This is greater than 1 if $p\geq 5$ (since then
$p-1\geq 4$). By the previous paragraph's exact sequence,
it follows that the Euler characteristic $\chi(G/P,A)$ 
is negative for $p\geq 5$. Since $A$ has no cohomology
in degrees at least 2, we must have
$H^1(G/P,A)\neq 0$, as we want.

For $p=3$, $h^0(G/B,\lambda)$ is 224 whereas $h^0(G/B,\lambda-\alpha)$
is 175, and so the dimensions would allow
the $G$-linear map $\varphi\colon H^0(G/B,\lambda)
\arrow H^0(G/B,\lambda-\alpha)$ to be surjective. But in fact
it is not surjective (and hence $H^1(G/P,A)$ is not zero),
as we now show.

For a dominant weight $\mu$, write $L(\mu)$ for the simple
$G$-module with highest weight $\mu$.
For a reductive group $G$ in any characteristic and a dominant weight $\mu$,
Chevalley showed
that the socle
(maximal semisimple submodule) of the Schur module $H^0(G/B,\mu)$
is simple, written $L(\mu)$. Moreover, this construction gives a one-to-one
correspondence between the simple $G$-modules and the dominant weights
\cite[Corollary II.2.7]{Jantzen}.

The Steinberg tensor product theorem describes all simple
$G$-modules in terms of those whose highest weight has
coefficients less than $p$ \cite[Corollary II.3.17]{Jantzen}.
In particular, since $\lambda=p\, \omega_1+\omega_{2}$, the theorem
says that
$$L(\lambda)\cong L(\omega_{1})^{[1]}\otimes L(\omega_{2}),$$
writing $M^{[1]}$ for the Frobenius twist of a $G$-module $M$.
For $SL(n)$ in any characteristic,
the simple module associated to the fundamental weight
$\omega_i$ is the exterior power $\Lambda^i(V)$ with $V$ the standard
$n$-dimensional representation \cite[section II.2.15]{Jantzen}.
So $L(\lambda)$ has dimension $\binom{n}{2}n=\binom{p+2}{2}(p+2)$.
Also, because $\lambda>\lambda-\alpha$ in the partial ordering
of the weight lattice given by the positive roots, the weight
$\lambda$ does not occur in the $G$-module $H^0(G/B,\lambda-\alpha)$.
It follows that the $G$-linear map $\varphi\colon
H^0(G/B,\lambda)\arrow H^0(G/B,\lambda-\alpha)$ is zero
on the simple submodule $L(\lambda)$.

For $p=3$, $H^0(\lambda)/L(\lambda)$ has dimension $224-50=174$,
whereas $H^0(\lambda-\alpha)$ has dimension 175. It follows
that $\varphi$ is not surjective. Equivalently, $H^1(G/P,A)\neq 0$,
as we want.
\end{proof}

\section{Lower-dimensional failure of Kodaira vanishing
for Fano varieties}

In this section, we give slightly lower-dimensional
examples of smooth Fano varieties $X$ in any characteristic $p>0$
for which Kodaira vanishing fails: dimension $2p-1$ rather
than $2p+1$ for $p\geq 3$, and dimension 5 for $p=2$.
(In particular, the examples in characteristics 2 or 3 have
dimension 5, which is smaller than the dimension 6 of Haboush-Lauritzen's
smooth Fano variety in characteristic 2 where Kodaira vanishing fails.)
In return for this improvement, we consider ample
line bundles that are not rational multiples of $-K_X$.

\begin{theorem}
\label{van2p-1}
Let $p$ be a prime number. Then there is a smooth
Fano variety $X$ over $\F_p$ and a very ample line bundle $A$
on $X$ such that $H^1(X,K_X+A)\neq 0$. Here $X$ has dimension $2p-1$
for $p\geq 3$ and dimension 5 for $p=2$. Also, $X$ has
Picard number 2.

Moreover, in the examples with $p\neq 3$,
the Euler characteristic $\chi(X,K_X+A)$
is negative.
\end{theorem}

\begin{proof}
First assume $p\geq 3$. At the end, we will give the example
for $p=2$.

As in the proof of Theorem \ref{van2p+1}, let $X$ be
the homogenous variety over $\F_p$ defined in
section \ref{homvar}, but now with $n$ equal to $p+1$ rather than
$p+2$. Thus $X$ is a smooth projective homogenous variety for $SL(n)$
of dimension $2n-3=2p-1$. The anticanonical bundle of $X$
is 
\begin{align*}
-K_X&=2p\,\omega_1+(n-p)\omega_{2}\\
&=2p\,\omega_1+\omega_{2}.
\end{align*}
Because $2p$ and 1 are positive, $-K_X$ is ample.
(In this case, $-K_X$ is not divisible by 2.)

The Picard group of $X$ has a basis consisting of $p\,\omega_1$
and $\omega_2$. Therefore,
the weight $\lambda:=3p\,\omega_1+\omega_{2}$ corresponds
to another
ample line bundle $A$ on $X$. In fact, $A$ is very ample, as discussed
in section \ref{homvar}.
The weight $\mu:=K_X+\lambda$ is equal to
$p\, \omega_{1}$. Because the simple root $\alpha:=\alpha_{1}=
2\omega_1-\omega_2$
has $\langle
\mu,\alpha^{\vee}\rangle=p$ and $\langle
\mu-\alpha,\alpha^{\vee}\rangle=p-2$ (which is less than $p$),
the same argument
as in the proof of Theorem \ref{van2p+1} gives an exact
sequence of $G$-modules:
$$0\arrow H^0(G/P,K_X+A)\arrow H^0(G/B,\mu)\arrow
H^0(G/B,\mu-\alpha)\arrow H^1(G/P,K_X+A)\arrow 0,$$
and $H^i(G/P,K_X+A)=0$ for $i\geq 2$.

Write $V$ for the $n$-dimensional representation $H^0(G/B,\omega_{1})$
of $G=SL(n)$. Since $\mu=p\, \omega_{1}$, $H^0(G/B,\mu)$
is the symmetric power $S^p(V)$, which has dimension
$\binom{n+p-1}{n-1}=\binom{2p}{p}$. Also, $\mu-\alpha$ is equal to
$(p-2)\omega_1+\omega_{2}$, which is also dominant,
and the Weyl dimension formula gives that $h^0(G/B,
\mu-\alpha)$ is
$\binom{2p-1}{p-1}(p-1)$. It follows that the ratio
$h^0(G/B,\mu-\alpha)/h^0(G/B,\mu)$ is equal to $(p-1)/2$.
Now suppose that $p\geq 5$. Then $H^0(G/B,\mu-\alpha)$ has bigger dimension
than $H^0(G/B,\mu)$. By the exact sequence above,
it follows that the Euler characteristic $\chi(X,K_X+A)$ is negative.
Since $K_X+A$ has no cohomology in degrees at least 2, it follows
that $H^1(X,K_X+A)\neq 0$ for $p\geq 5$, as we want.

Next, let $p=3$. Then $G=SL(4)$, and the $G$-modules $H^0(G/B,\mu)=S^3V$
and $H^0(G/B,\mu-\alpha)$ both have
dimension 20. However, because $\mu>\mu-\alpha$ in the partial order
of the weight lattice
given by the positive roots, the weight $\mu$ occurs
in $H^0(G/B,\mu)$ and not in $H^0(G/B,\mu-\alpha)$.
Therefore,
the $G$-linear map $\varphi\colon H^0(G/B,\mu)
\arrow H^0(G/B,\mu-\alpha)$ is not an isomorphism,
and hence not surjective. By the exact sequence above,
it follows that $H^1(X,K_X+A)\neq 0$.

Finally, let $p=2$. In this case, let $n$ be $p+2=4$ (not $p+1$
as above). Let $X$ be the homogeneous variety
for $SL(n)$ over $\F_2$ described in section \ref{homvar}.
Then $X$ is a smooth Fano variety of dimension $2n-3=5$ and Picard number 2.

The Picard group of $X$ is generated by $p\, \omega_1=2\omega_1$
and $\omega_{2}$. The anticanonical bundle $-K_X$ is
$2p\,\omega_1+(n-p)\omega_{2}=4\omega_1+2\omega_{2}$.
Let $A$ be the ample line bundle $6\omega_1+\omega_{2}$ on $X$.
It is in fact very ample, as discussed in section \ref{homvar}.
Let $\mu=K_X+A=2\omega_1-\omega_{2}$.
Because of the negative coefficient, $H^0(X,\mu)=0$ (for example
by the inclusion $H^0(X,\mu)\subset H^0(G/B,\mu)$ given
by pullback). So Kodaira vanishing would imply that $H^i(X,\mu)=0$
for all $i$.

To disprove that, we compute the Euler characteristic. As in the proof
of Theorem \ref{van2p+1}, $X$ is a $\P^1$-bundle over the Grassmannian
$\Gr(2,n)$.
The degree of a line bundle $\nu$ on the $\P^1$ fibers
is $\langle \nu,\alpha^{\vee}\rangle/p$,
where $\alpha$ is the simple root $\alpha_{1}$
and $p=2$. So $\mu$ has degree
1 on those fibers. As in the proof of Theorem \ref{van2p+1},
it follows that there is a long exact sequence of $G$-modules:
$$\arrow H^i(G/B,\mu-p\alpha)\arrow H^i(X,\mu) \arrow 
H^i(G/B,\mu)\arrow H^{i+1}(G/B,\mu-p\alpha)\arrow $$
Therefore, in terms of Euler characteristics,
$$\chi(X,\mu)=\chi(G/B,\mu)+\chi(G/B,\mu-p\alpha).$$

Here $G=SL(4)$ and $\mu=2\omega_1-\omega_2$, and so
$\mu-p\alpha=-2\omega_1+\omega_2$. The Weyl dimension formula
says that
\begin{multline*}
\chi(G/B,a_1\omega_1+a_2\omega_2+a_3\omega_3)\\
=\frac{(a_1+1)
(a_2+1)(a_3+1)(a_1+a_2+2)(a_2+a_3+2)(a_1+a_2+a_3+3)}{12}.
\end{multline*}
It follows that $\chi(G/B,\mu)=0$ and $\chi(G/B,\mu-p\alpha)=-1$,
and hence $\chi(X,\mu)=-1$. Because this is negative, Kodaira
vanishing fails on the smooth Fano 5-fold $X$ in characteristic 2.

To prove the full statement of Theorem \ref{van2p-1}, we want
to show more specifically that $H^1(X,K_X+A)=H^1(X,\mu)$
is not zero.
It suffices to show that $H^i(X,\mu)=0$ for $i>1$.
By the exact sequence above, this follows if the line bundles
$\mu=2\omega_1-\omega_{2}$ and $\mu-p\alpha=
-2\omega_1+\omega_2$ on $G/B$ have no cohomology in degrees
greater than 1. Because $\langle\mu,\alpha_2^{\vee}
\rangle=-1$ (that is, $\mu$ has degree $-1$ on the fibers
of one of the $\P^1$-fibrations of $G/B$), $\mu$ actually has
no cohomology in any degree \cite[Proposition II.5.4(a)]{Jantzen}.
Next, the trivial line bundle on $G/B$ has cohomology only in degree
zero by Kempf's vanishing theorem, and $s_{\alpha}\cdot 0
= -2\omega_1+\omega_2=\mu-p\alpha$. Beause $\langle 
0,\mu^{\vee}\rangle = 0$ is of the form $sp^m-1$
for some $s,m\in \N$ with $0<s<p$, Theorem \ref{andersen}
gives that
$$H^i(G/B,O)\cong H^{i+1}(G/B,\mu-p\alpha)$$
for all $i$. Thus $\mu-p\alpha$ has cohomology only
in degree 1, as we want.
\end{proof}

\section{Quotient singularities}
\label{quotient}

In this section, we describe Yasuda's examples
of quotient singularities of any characteristic
$p>0$ which are terminal but not Cohen-Macaulay (section
\ref{quotient}). Again, the dimension
increases with $p$, but more slowly than in the examples above.

For $p\geq 5$, his construction is as follows. Let $G$ be the cyclic
group $\Z/p$ and $k$ a field of characteristic $p$. For each
$1\leq n\leq p$, there is a unique indecomposable representation
$V$ of $G$ over $k$ of dimension $n$, with a generator of $G$
acting by a single Jordan block. Assume that 
$n(n-1)/2> p\geq n\geq 4$. By \cite[Proposition 6.9]{Yasuda},
$X:=V/G$ is klt if $n(n-1)/2\geq p$ (and the converse would
follow from resolution of singularities). An extension of the
arguments yields that $X$ is terminal if $n(n-1)/2>p$.
On the other hand, because the fixed point set $V^G$ has dimension 1,
which has codimension at least 3 in $V$, $X$ is not
Cohen-Macaulay, by Fogarty \cite{Fogarty}. By a similar construction (using
decomposable representations of $\Z/p$), Yasuda finds non-Cohen-Macaulay
terminal quotient singularities of dimension 6 in characteristic 2
and of dimension 5 in characteristic 3.

\section{A 3-dimensional terminal singularity that is not Cohen-Macaulay}

Inspired by Yasuda's examples from section \ref{quotient},
we now give the first example of a terminal 3-fold singularity
which is not Cohen-Macaulay. The base field can be taken
to be any field of characteristic 2, say $\F_2$. Write $G_m=A^1-\{0\}$
for the multiplicative group.

\begin{theorem}
\label{dim3}
Let $X$ be the 3-fold $(G_m)^3/(\Z/2)$ over the field
$\F_2$, where the generator
$\sigma$ of $\Z/2$ acts by
$$\sigma(x_1,x_2,x_3)=\bigg( \frac{1}{x_1},\frac{1}{x_2},\frac{1}{x_3}\bigg).$$
Then $X$ is terminal but not Cohen-Macaulay.
\end{theorem}

\begin{proof}
Let $Y$ be the 3-fold $(G_m)^3$ over the field $k=\F_2$.
Clearly $G=\Z/2$ acts freely
outside the point $(1,1,1)$ in $Y$. By Fogarty, when
the group $G=\Z/p$ acts on a regular scheme $Y$ in characteristic $p$
with fixed point locus nonempty and of codimension at least 3, $Y/G$
is not Cohen-Macaulay \cite{Fogarty}.

It remains to show that $X=Y/G$ is terminal. Most of the work
is to construct an explicit resolution of singularities
of $X$. We do that by performing $G$-equivariant blow-ups
of $Y$ until the quotient variety becomes smooth over $k$.
In characteristic 0, the quotient of a smooth variety by a cyclic group
of prime order
is smooth if and only if the fixed point locus has pure codimension 1.
That fails for actions of $\Z/p$ in characteristic $p$, but there
is a useful substitute by Kir\'aly and L\"utkebohmert \cite[Theorem 2]{KL}:

\begin{theorem}
\label{kl}
Let $G$ by a cyclic group of prime order which acts on a regular scheme
$X$. If the fixed point {\it scheme }$X^G$ is a Cartier divisor in $X$, then
the quotient space $X/G$ is regular.
\end{theorem}

Since we work with varieties
over the perfect field $k=\F_2$, being regular is the same as
being smooth over $k$.

In checking the properties of the blow-up, it is helpful
to observe that the singularity $X=Y/G$ has an enormous
automorphism group. Namely, $GL(3,\Z)$ acts by automorphisms
of the torus $Y=(G_m)^3$, and this commutes with the action
of $G=\Z/2$ (which corresponds to the diagonal matrix $-1$ in $GL(3,\Z)$).
Therefore, $GL(3,\Z)$ acts on $X$ (through its quotient $PGL(3,\Z)$,
clearly). The formal completion of $X$ at its singular point
has an action of an even bigger group, $PGL(3,\Z_2)$.

\begin{remark}
By analogy with the study of infinite
discrete automorphism groups of projective varieties,
this example suggests that it could be interesting to study
profinite groups acting on formal completions of singularities
in characteristic $p$.
\end{remark}

We now begin to blow up. Identify $Y$ with $Y_0=(A^1-\{1\})^3$
over $k$ by $y_i=x_i+1$ for $i=1,2,3$. Then $G$ acts on $Y_0$ by 
$$\sigma(y_1,y_2,y_3)=\bigg(\frac{y_1}{1+y_1},
\frac{y_2}{1+y_2},\frac{y_3}{1+y_3}\bigg).$$
We do this so that the point fixed by $G$ is $(y_1,y_2,y_3)=(0,0,0)$.
Let $Y_1$ be the blow-up of $Y_0$ at this point.
Thus
$$Y_1=\{((y_1,y_2,y_3),[w_1,w_2,w_3])\in Y_0\times\P^2: y_1w_2=y_2w_1,
y_1w_3=y_3w_1,y_2w_3=y_3w_2\}.$$
The group $G$ acts on $Y_1$ by
$$\sigma((y_1,y_2,y_3),[w_1,w_2,w_3])=\bigg(\frac{y_1}{1+y_1},
\frac{y_2}{1+y_2},\frac{y_3}{1+y_3}\bigg),\bigg[\frac{w_1}{1+y_1},
\frac{w_2}{1+y_2},\frac{w_3}{1+y_3}\bigg].$$
Because $GL(3,\Z)$ fixes the origin in $Y_0$, the action
of $GL(3,\Z)$ on $Y_0$ lifts to an action on the blow-up $Y_1$.
(This includes the obvious action of the symmetric group $S_3$.)

To compute the fixed point scheme of $G$ on $Y_1$, work in the open
subset $U_1$ with $w_1=1$; this will suffice, by the $S_3$-symmetry of $Y_1$.
We can view $U_1$ as the open subset of $A^3=\{(y_1,w_2,w_3)\}$
defined by $y_1\neq 1$, $y_1w_2\neq 1$, and $y_1w_3\neq 1$
(using that $y_2=y_1w_2$ and $y_3=y_1w_3$).
The action of $G$ on $U_1$ is given by
$$\sigma(y_1,w_2,y_3)=\bigg(\frac{y_1}{y_1+1},\frac{w_2(y_1+1)}{y_1w_2+1},
\frac{w_3(y_1+1)}{y_1w_3+1}\bigg).$$
So the fixed point scheme $(U_1)^G$ is defined by the equations
$$y_1=\frac{y_1}{y_1+1},\; w_2=\frac{w_2(y_1+1)}{y_1w_2+1}, \;
w_3=\frac{w_3(y_1+1)}{y_1w_3+1}.$$
Equivalently, $y_1^2=0$, $y_1w_2(w_2+1)=0$,
and $y_1w_3(w_3+1)=0$. Thus the scheme $(U_1)^G$ is not a Cartier
divisor: it is equal to the Cartier divisor $y_1=0$ (the exceptional
divisor) except at the points $(y_1,w_2,w_3)=(0,0,0),(0,0,1),(0,1,0)$,
and $(0,1,1)$.

In view of the $S_3$-symmetry of $Y_1$, it follows that the fixed
point scheme $(Y_1)^G$ is equal to the exceptional divisor $E\cong \P^2$
with multiplicity 1
except at 7 points on that divisor, where $(y_1,y_2,y_3)=(0,0,0)$
and $[w_1,w_2,w_3]$ is one of $[1,0,0]$, $[0,1,0]$,
$[0,0,1]$, $[1,1,0]$, $[1,0,1]$, $[1,1,0]$, or $[1,1,1]$.
Note that $GL(3,\Z)$ acts through its quotient group
$GL(3,\F_2)$ on the divisor $E$, and it permutes these 7 points
transitively. Therefore, to resolve the singularities
of $Y_1/G$, it will suffice to blow $Y_1$ up
at the point $[1,0,0]$; the blow-ups at the rest of the 7 points
work exactly the same way.

The point we are blowing up is in the open set
$U_1$ of $Y_1$ defined above, namely $(y_1,w_2,w_3)=(0,0,0)$
in 
$$U_1=\{(y_1,w_2,w_3)\}\in A^3:
y_1\neq 1, \; y_1w_2\neq 1, \text{ and }y_1w_3\neq 1\}.$$
The resulting blow-up $Y_2$ is:
$$\{((y_1,w_2,w_3),[v_1,v_2,v_3])\in U_1\times\P^2: y_1v_2=w_2v_1,
y_1v_3=w_3v_1, w_2v_3=w_3v_2\}.$$
The group $G$ acts on $Y_2$ by:
\begin{multline*}
\sigma((y_1,w_2,w_3),[v_1,v_2,v_3])=\\
\bigg(\frac{y_1}{y_1+1},\frac{w_2(y_1+1)}{y_1w_2+1},
\frac{w_3(y_1+1)}{y_1w_3+1}\bigg),\bigg[\frac{v_1}{y_1+1},
\frac{v_2(y_1+1)}{y_1w_2+1},
\frac{v_3(y_1+1)}{y_1w_3+1}\bigg].
\end{multline*}

To compute the fixed point scheme of $G$ on $Y_2$, work first
in the open set $v_1=1$. In those coordinates, $G$ acts by:
$$\sigma(y_1,v_2,v_3)=\bigg(\frac{y_1}{y_1+1},
\frac{v_2(y_1+1)^2}{y_1^2v_2+1},\frac{v_3(y_1+1)^2}{y_1^2v_3+1}\bigg)$$
(using that $w_2=y_1v_2$ and $w_3=y_1v_3$).
The fixed point scheme of $G$ on this open set is given by the equations
$y_1^2=0$, $y_1^2v_2(v_2+1)=0$, and $y_1^2v_3(v_3+1)=0$,
which just say that $y_1^2=0$. Thus the fixed point scheme $(Y_2)^G$
is a Cartier divisor in this open set: 2 times the new exceptional divisor
$E_1$. 

We next compute the fixed point scheme $(Y_2)^G$ in the open set
$v_2=1$;
we will not need to consider the remaining open set $v_3=1$ separately,
in view of the symmetry between $v_2$ and $v_3$ in the action of $G$ on $Y_2$.
Namely, $G$ acts on the open set $v_2=1$ by
$$\sigma(w_2,v_1,v_3)=\bigg(\frac{w_2(w_2v_1+1)}{w_2^2v_1+1},
\frac{v_1(w_2^2v_1+1)}{(w_2v_1+1)^2},
\frac{v_3(w_2^2v_1+1)}{w_2^2v_1v_3+1}\bigg)$$
(using that $y_1=w_2v_1$ and $w_3=w_2v_3$).
The fixed point scheme of $G$ on this open set is given by the equations
$$w_2^2v_1(w_2+1)=0,\; w_2^2v_1^2(v_1+1)=0,\text{ and }w_2^2v_1v_3(v_3+1)=0.$$
Note that we defined $Y_2$ by blowing up only one of the 7 points
listed earlier in $Y_1$, the one with $(y_1,w_2,w_3)=(0,0,0)$; since
we are not concerned with the other 6 points here, we can assume that
$w_2\neq 1$. Then the first equation defining $(Y_2)^G$
gives that $w_2^2v_1=0$, and that implies the other two equations.
That is, we have shown that $(Y_2)^G$ is a Cartier divisor
in the open set $v_2=1$ near the exceptional divisor $E_1$:
it is $E_0+2E_1$, where $E_0$ is the proper transform
of the exceptional divisor $E$ in $Y_1$ (given by $v_1=0$ in these
coordinates).

By the symmetry between $v_2$ and $v_3$ in the equations
for $Y_2$, the same calculation applies to the open set
$v_3=1$. Thus we have shown that the fixed point scheme
$(Y_2)^G$ is Cartier near the exceptional divisor
$E_1$.

From now on, write $Y_2$ for the blow-up of $Y_1$ at all 7 points
listed above. By the previous calculation together with
the $GL(3,\Z)$-symmetry of $Y_2$, the fixed point scheme
$(Y_2)^G$ is the Cartier divisor
$$E_0+2\sum_{j=1}^7 E_j,$$
where $E_1,\ldots,E_7$ are the 7 exceptional divisors
of $Y_2\arrow Y_1$. By Theorem \ref{kl},
it follows that $Y_2/G$ is smooth over $k$.
Thus $Y_2/G$ is a resolution of singularities of $X=Y_0/G$.

Write $F_0,F_1,\ldots,F_7$ for the images in $Y_2/G$ of
the exceptional divisors $E_0,E_1,\ldots,E_7$. Note that although $G$
fixes each divisor $E_j$ in $Y_2$, the morphism $E_j\arrow F_j$
is a finite purely inseparable morphism, not necessarily an isomorphism.
(Indeed, $G=\Z/2$ is not linearly reductive in characteristic
2. So if $G$ acts on an affine scheme $T$ preserving a closed subscheme $S$,
the morphism $S/G\arrow T/G$ need not be a closed immersion. That is,
the $G$-equivariant surjection $O(T)\arrow O(S)$ need not yield
a surjection $O(T)^G\arrow O(S)^G$.) In any case, our construction
shows that the dual graph of the resolution $Y_2/G\arrow X$ is a star,
with one edge from the vertex $F_0$ to each of the other 7 vertices
$F_1,\ldots,F_7$. 
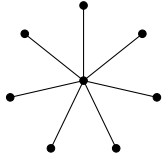
\begin{center}
\begin{tikzpicture}
\draw (0,0) -- (0,1);
\draw (0,0) -- (0.78183,0.62349);
\draw (0,0) -- (0.97493,-0.22252);
\draw (0,0) -- (0.43389,-0.90097);
\draw (0,0) -- (-0.43389,-0.90097);
\draw (0,0) -- (-0.97493,-0.22252);
\draw (0,0) -- (-0.78183,0.62349);
\draw[fill] (0,0) circle [radius=0.05];
\draw[fill] (0,1) circle [radius=0.05];
\draw[fill] (0.78183,0.62349) circle [radius=0.05];
\draw[fill] (0.97493,-0.22252) circle [radius=0.05];
\draw[fill] (0.43389,-0.90097) circle [radius=0.05];
\draw[fill] (-0.43389,-0.90097) circle [radius=0.05];
\draw[fill] (-0.97493,-0.22252) circle [radius=0.05];
\draw[fill] (-0.78183,0.62349) circle [radius=0.05];
\end{tikzpicture}
\end{center}
This generalizes Artin's observation
that the analogous singularity one dimension lower,
$(G_m)^2/(\Z/2)$ in characteristic 2, is a $D_4$ surface singularity.
That is, the dual graph of its minimal resolution is again a star, with one
central vertex connected to 3 other vertices \cite[p.~64]{Artin}.

The divisor class $K_X$ is Cartier on $X=Y_0/G$, because $G$ preserves
the volume form $(dx_1/x_1)\wedge (dx_2/x_2)\wedge
(dx_3/x_3)$ on the torus $Y_0\cong (G_m)^3$. So we can write
$$K_{Y_2/G}=\pi^*K_X+\sum_{j=0}^7 a_jF_j$$
for some (unique) integers $a_i$. The variety $X$ is terminal if and only
if the {\it discrepancy }$a_i$ is positive
for all $i$ \cite[Corollary 2.12]{Kollarsing}.
Here and below, we write $\pi$ for all the relevant contractions,
which in the formula above means $\pi\colon Y_2/G\arrow Y_0/G=X$.

The analogous formula for $Y_2$ is easy, because $Y_2$ is obtained
from $Y_0$ by blowing up points. First, since $Y_1$
is the blow-up of the smooth 3-fold $Y_0$ at a point,
$$K_{Y_1}=\pi^*K_{Y_0}+2E.$$
Next, $Y_2$ is the blow-up of $Y_1$ at 7 points on the exceptional
divisor $E$, and so
\begin{align*}
K_{Y_2}&=\pi^*K_{Y_1}+2\sum_{j=1}^7 E_i\\
&=\pi^*K_{Y_0}+2E_0+4\sum_{j=1}^7 E_i,
\end{align*}
using that $\pi^*E=E_0+\sum_{j=1}^7 E_i$.

Write $f$ for the quotient map $Y_0\arrow Y_0/G$ or $Y_2\arrow Y_2/G$.
It remains to compute
the {\it ramification index }of each divisor $E_j$ in $Y_2$
(the positive integer
$e_j$ such that $f^*F_j=e_jE_j$) and the coefficient $c_j$ of $E_j$
in the ramification divisor (meaning that $K_{Y_2}=f^*K_{Y_2/G}+
\sum_j c_jE_j$). Another name for $c_j$
is the {\it valuation of the different }$v_L(\mathcal{D}_{L/K})$,
where $L$ is the function field $k(Y_2)$, $K=k(Y_2/G)$, and $v_L$
is the valuation of $L$ associated to the divisor $E_j$.
Here $e_jf_j=2$, where $f_j$ is the degree of the field extension
$k(E_j)$ over $k(F_j)$ (which is purely inseparable in the case
at hand).

We want to compute these numbers without actually finding equations
for the quotient variety $Y_2/G$. This can be done using the
{\it Artin }and {\it Swan ramification numbers} of the $G$-action
on $Y_2$, defined as:
\begin{align*}
i(\sigma)&=\inf_{a\in O_L} v_L(\sigma(a)-a)\\
s(\sigma)&=\inf_{a\in L^*} v_L(\sigma(a)a^{-1}-1).
\end{align*}
Here $O_L$ is the ring of integers of $L=k(Y_2)$ with respect
to the valuation $v_L$ associated to a given divisor $E_j$.
We have already computed $i(\sigma)$ for each $E_j$: it is
the multiplicity of $E_j$ in the fixed point scheme $(Y_2)^G$,
which is 1 for $E_0$ and 2 for $E_j$ with $1\leq j\leq 7$.
Then, more generally for an action of $\Z/p$
on a normal scheme of characteristic $p$ that fixes an irreducible divisor,
we have $s(\sigma)>0$, and either $i(\sigma)=s(\sigma)+1$,
in which case $e=p$ and $f=1$,
or $i(\sigma)=s(\sigma)$,
in which case $e=1$ and $f=p$ \cite[section 2.1]{XZ}.
The first case is called {\it wild }ramification,
and the second is called {\it fierce} ramification.
In both cases, the valuation of the different $v_L(\mathcal{D}_{L/K})$
is equal to $(p-1)i(\sigma)$.

In particular, returning to our example with $p=2$,
we have computed $i(\sigma)$ for each divisor $E_j$
(the multiplicity of $E_j$
in the fixed point scheme $(Y_2)^G$), and (by the formula above
for $v_L(\mathcal{D}_{L/K})$) this computes the ramification
divisor of $f\colon Y_2\arrow Y_2/G$. Namely,
$$K_{Y_2}=f^*K_{Y_2/G}+E_0+2\sum_{j=1}^7 E_j.$$

The next step is to compute the ramification index of $f$
along each exceptional divisor $E_j$.
For $E_0$, we have $i(\sigma)=1$ (the multiplicity of $E_0$
in the fixed point scheme $(Y_2)^G$), and then the results above imply
that $Y_2\arrow Y_2/G$ is fiercely ramified along $E_0$,
and so $s(\sigma)=1$.
In particular, $e_0=1$, meaning that $f^*F_0=E_0$.

For $E_j$ with $1\leq j\leq 7$, we have $i(\sigma)=2$ (the multiplicity
of $E_j$ in the fixed point scheme $(Y_2)^G$), which implies
that $s(\sigma)$ is 1 or 2 by the results above. To resolve this ambiguity,
note that it suffices to compute $s(\sigma)$ for $E_1$, because
the automorphism group $GL(3,\Z)$ of $Y_2$ (commuting with $G$)
permutes $E_1,\ldots,E_7$ transitively. And we showed that $E_1$
is defined in the coordinate chart $v_1=1$
by the equation $y_1=0$,
on which $G$ acts by $\sigma(y_1)=y_1/(y_1+1)$.
So $v_L(\sigma(y_1)y_1^{-1}-1)
=v_L(y_1/(y_1+1))=1$. Since $s(\sigma)=
\inf_{a\in L^*} v_L(\sigma(a)a^{-1}-1)$, it follows that
$s(\sigma)$ is 1, not 2. Thus $Y_2\arrow Y_2/G$ is wildly
(rather than fiercely) ramified along $E_j$ for $1\leq j\leq 7$.
In particular, $f^*F_j=2E_j$.

Thus, we have shown that
$$K_{Y_2}=f^*K_{Y_2/G}+E_0+2\sum_{j=1}^7 E_j$$
and that $f^*F_0=E_0$ and $f^*F_j=2E_j$ for $1\leq j\leq 7$.
Since $f\colon Y_0\arrow Y_0/G$ is \'etale in codimension 1,
we have $K_{Y_0}=f^*K_{Y_0/G}$. It follows that
\begin{align*}
f^*K_{Y_2/G}&=K_{Y_2}-E_0-2\sum_{j=1}^7 E_j\\
&= \bigg(\pi^*K_{Y_0}+2E_0+4\sum_{j=1}^7E_j\bigg)-E_0-2\sum_{j=1}^7 E_j\\
&= \pi^*f^*K_{Y_0/G}+E_0+2\sum_{j=1}^7 E_j\\
&= f^*\bigg(\pi^*K_{Y_0/G}+F_0+\sum_{j=1}^7F_j bigg).
\end{align*}
Therefore, 
$$K_{Y_2/G}=\pi^*K_{Y_0/G}+F_0+\sum_{j=1}^7F_j.$$
Because the coefficient of every exceptional divisor
$F_j$ is positive, and $Y_2/G$
is a resolution of singularities of $Y_0/G$,
$X=Y_0/G$ is terminal.
\end{proof}

\section{Open questions}

One question suggested by these examples is whether,
for each positive integer $n$, there is a number $p_0(n)$
such that Fano varieties of dimension $n$ in characteristic
$p\geq p_0(n)$ satisfy Kodaira vanishing. (One could
ask this for smooth Fanos, or in greater generality.)
This is related to the fundamental
question of whether the smooth Fano varieties
of given dimension
form a bounded family over $\Z$, as they do in characteristic zero
by Koll\'ar-Miyaoka-Mori \cite[Corollary V.2.3]{Kollarrat}.

A related question is whether, for each positive integer $n$,
there is a number $p_0(n)$ such that klt singularities
of dimension $n$ in characteristic $p\geq p_0(n)$ are Cohen-Macaulay.
This was shown by Hacon and Witaszek for $n=3$, although no explicit
value for $p_0(3)$ is known \cite[Theorem 1.1]{HW}.


\small \sc UCLA Mathematics Department, Box 951555,
Los Angeles, CA 90095-1555

totaro@math.ucla.edu
\end{document}